\documentclass[a4paper,12pt,eqno]{article}
\usepackage{amsthm}
\usepackage{latexsym,amssymb,amsfonts,amsmath}
\usepackage{graphicx}
\usepackage[misc]{ifsym} 
\newtheorem{thm}{Theorem}[section]
\newtheorem{lem}[thm]{Lemma}
\newtheorem{cor}[thm]{Corollary}
\newtheorem{pro}[thm]{Proposition}
\newtheorem{rem}[thm]{Remark}

\newcommand{\RM}{\mathbb{R}}

\newcommand{\QM}{\mathbb{Q}}
\newcommand{\NM}{\mathbb{N}}
\newcommand{\CM}{\mathbb{C}}

\newcommand{\HM}{\mathbb{H}}
\newcommand{\Mat}{\operatorname{Mat}}
\newcommand{\Sdet}{\operatorname{Sdet}}
\newcommand{\RP}{\operatorname{Re}\,}

\title{{\Large {\bf The quaternionic weighted zeta function of a graph}}
\author{
{\small Norio Konno}\\
{\scriptsize Department of Applied Mathematics, 
Faculty of Engineering, 
Yokohama National University}\\
{\scriptsize Hodogaya, Yokohama 240-8501, Japan}\\
{\scriptsize e-mail: konno@ynu.ac.jp}\\
{\small Hideo Mitsuhashi}\\
{\scriptsize Faculty of Education, 
Utsunomiya University}\\
{\scriptsize Utsunomiya, Tochigi 321-8505, Japan}\\
{\scriptsize e-mail: mitsu@cc.utsunomiya-u.ac.jp}\\
{\small Iwao Sato}\\
{\scriptsize Oyama National College of Technology}\\
{\scriptsize Oyama, Tochigi 323-0806, Japan}\\
{\scriptsize e-mail: isato@oyama-ct.ac.jp}\\}
}
\date{\empty }
\begin{document}
\maketitle

\par\noindent
\begin{small}
\par\noindent
{\bf Abstract}. 
We establish the quaternionic weighted zeta function of a graph and 
its Study determinant expressions. 
For a graph with quaternionic weights on arcs, we 
define a zeta function by using an infinite product which is regarded as 
the Euler product. 
This is a quaternionic extension of the square of the Ihara zeta function. 
We show that the new zeta function can be expressed as the exponential of a 
generating function and that it has two Study determinant expressions, which are 
crucial for the theory of zeta functions of graphs.

\footnote[0]{
{\it Abbr. title:} The quaternionic weighted zeta function of a graph
}
\end{small}

\setcounter{equation}{0}
\section{Introduction}
Zeta functions of graphs started from Ihara zeta functions of regular 
graphs by Ihara \cite{Ihara1966}. 
Originally, Ihara \cite{Ihara1966} presented $p$-adic Selberg zeta functions of 
discrete groups, and showed that its reciprocal is an explicit polynomial. 
Ihara also showed the logarithm of Ihara zeta function has an expression in the form of 
a generating function. 
Serre \cite{Serre} pointed out that the Ihara zeta function is the zeta function of 
the quotient $T/ \Gamma $ (a finite regular graph) of the one-dimensional 
Bruhat-Tits building $T$ (an infinite regular tree) associated with 
$GL(2, k_p)$. 
A zeta function of a regular graph $G$ associated with a unitary 
representation of the fundamental group of $G$ was developed by 
Sunada \cite{Sunada1986}, \cite{Sunada1988}. 
Hashimoto \cite{Hashimoto1989} treated multivariable zeta functions of bipartite graphs. 
For a general graph, Hashimoto \cite{Hashimoto1989} gave a determinant expression 
for its Ihara zeta function by its edge matrix. 
Bass \cite{Bass1992} generalized Ihara's result on the Ihara zeta function of 
a regular graph to an irregular graph, and showed that its reciprocal is 
again a polynomial. 
Various proofs of Bass' Theorem have been given by 
Stark and Terras \cite{ST1996}, Foata and Zeilberger \cite{FZ1999}, Kotani and Sunada \cite{KS2000}. 
For the weighted type of zeta function of graphs, Hashimoto \cite{Hashimoto1990} 
introduced a zeta function of a graph with weights assigned to its edges. 
Stark and Terras \cite{ST1996} defined the edge zeta function of a graph with weights 
assigned to its arcs, and gave its determinant expression by using its edge matrix. 
Mizuno and Sato \cite{MS2004} introduced a special version of the edge zeta function 
of a graph, and defined the weighted zeta function of a graph by using arc weights 
and the variable $t$ counting the length of cycles. 
Later, they called this zeta function the first weighted zeta function in order to 
distinguish it from another zeta function which Sato defined after. 
It is noteworthy that the expressions of those zeta functions by the exponentials 
of generating functions were important to deriving the determinant expressions 
in the studies above. 
Recently, Watanabe and Fukumizu \cite{WF2009} presented a determinant expression 
for the edge zeta function of a graph by using matrices with size of the number of 
its vertices. 
As described above, zeta functions of graphs have been investigated for half a century. 

On the other hand, the quaternions were discovered by Hamilton in 1843. 
It can be considered as an extension of the complex numbers. However, 
quaternions do not commute mutually in general. 
For many years, a number of people, for example Cayley, 
Study, Moore, Dieudonn{\'e}, Dyson, Mehta, Xie, Chen, 
have given different definitions of determinants of quaternionic matrices. 
Detailed accounts on the determinants of  
quaternionic matrices can be found in, for example, \cite{Aslaksen1996,Zhang1997}. 
In order to extend the zeta function of a graph to the case of quaternions, 
we will use the approach developed by Study \cite{Study1920}. 
The Study determinant is the unique, up to a real power factor, functional 
which satisfies the following three axioms \cite{Aslaksen1996}: 
\begin{description}
\item (A1) $d({\bf A})=0{\;\Leftrightarrow\;}{\bf A}$ is singular.
\item (A2) $d({\bf AB})=d({\bf A})d({\bf B})$ for all ${\bf A,B}{\in}\Mat(n,\HM)$.
\item (A3) If ${\bf A}'$ is obtained form ${\bf A}$ by adding a left-multiple of a row to 
another row or a right-multiple of a column to another column, 
then $d({\bf A}')=d({\bf A})$.
\end{description}
Therefore it is essential to investigate the relation between the Study determinant
and the quaternionic zeta function of a graph. 
The advantage of this approach is that one can reduce a calculation of the Study determinant 
to that of the ordinary determinant. Its disadvantage is that the Study determinant 
is not an exact extension of the determinant but is rather that of the square of it.

Our aim in this paper is to define the quaternionic zeta function of a graph 
by using an infinite product, namely the Euler product, and to show its essential characteristics. 
Specifically, we derive its expression in the form of the exponential of a generating function 
and determine its two types of Study determinant expressions such as in 
\cite{Hashimoto1989} (Hashimoto type) and in \cite{Bass1992} (Bass type). 
Both types are crucial for a zeta function of a graph. 
In order to obtain these results, we explain the formal power series on 
a monoid algebra for preparations. 
The formal series, which includes the formal power series, is treated in detail 
in \cite{BR2011}
Our approach in this material follows the manner in \cite{FZ1999,RS1987}. 
Furthermore, we give an account of some relations between their exponentials and 
logarithms which can be found in Chapter IV of \cite{Bourbaki}. 
Up to the present, no quaternionic extension of 
zeta functions of graphs has been appeared. Hence our results can be regarded as 
the first successful construction of a quaternionic zeta function of a graph.

The rest of the paper is organized as follows. 
Section 2 treats noncommutative formal series and monoids. 
The notion of Lyndon word is given, and some application of the factorization theorem 
is discussed. The conclusion (Proposition \ref{AmitsurForMatrixThm}) plays an 
essential role in our investigation. 
In Section 3, we define the exponential and the logarithm of noncommutative formal 
power series and give some formulas of which we make use in Section 6. 
In Section 4, we explain the Study determinant of a quaternionic matrix and extend it to 
the matrix whose entries are formal power series with quaternionic coefficients. 
In Section 5, we provide a summary of the Ihara zeta function and the various zeta function 
of a graph, and present their determinant expressions. 
In Section 6, we define a quaternionic zeta function of a graph by using an infinite 
product, and give its expression in the form of the exponential of a generating function. 
In Section 7, we show that the quaternionic zeta functions has two types of 
Study determinant expressions. 
One (Theorem \ref{DetExpressionByEdgeMatrix}) is an analogy to Hashimoto type expression 
\cite{Hashimoto1989} and 
another (Theorem \ref{DetExpressionOfBassType}) is to Bass type \cite{Bass1992}.

\section{Noncommutative formal power series and monoids}
Let $R$ be a commutative ring with unity, and $A$ an algebra over $R$. 
$A[[t]]$ denotes the ring of formal power series in $t$ 
with coefficients in $A$. 
Each element $\alpha$ of $A[[t]]$ is expressed as 
\[
\alpha=\sum_{k{\geq}0}\alpha_kt^k \quad (\alpha_k{\;\in\;}A).
\]
$A[[t]]$ can be equipped with the topology defined by 
the following manner. 
Let $\omega$ be the function defined as follows: 
\begin{equation*}\label{UltrametricDistance}
\begin{split}
\omega : &A[[t]]{\times}A[[t]]
{\longrightarrow}\NM{\;\cup\;}\{\infty\} \\
& \omega(\alpha,\beta)=\inf \{n{\;\in\;}\NM\;|\;\alpha_n{\;\neq\;}\beta_n \}.
\end{split}
\end{equation*}
Then an ultrametric distance $d_{\omega}$ on $A[[t]]$ 
is given by $d_{\omega}(\alpha,\beta)=2^{-\omega(\alpha,\beta)}$ and a topology on 
$A[[t]]$ is derived from $d_{\omega}$.

Let $G$ be a monoid. 
$R[G]$ denotes the {\it monoid algebra} of $G$ over $R$.  
$R[G]$ is the set of formal sums $z=\sum_{g{\;\in\;}G}z_gg$, 
where $z_g{\;\in\;}R$ for each $g{\;\in\;}G$ and $z_g=0$ for all 
but finitely many $g$. 
The addition in $R[G]$ is coefficient-wise, and 
the elements of $R$ commute with the elements of $G$ in the multiplication. 

Let $X=\{x_1,{\cdots},x_N\}$ be a finite nonempty totally ordered set in which 
elements are arranged ascendingly. 
$X^*$ denotes the free monoid generated by 
$X$. Let $<$ be the lexicographic order on $X^*$ derived from the total order on $X$. 
For a word $w=x_{i_1}x_{i_2}{\cdots}x_{i_r}{\;\in\;}X^*$, 
$r$ is called the {\it length} of $w$ which is denoted by $|w|$. 
The length of the empty word is defined to be $0$. 
A nonempty word $w$ in $X^*$ is called a {\it Lyndon word} if $w$ is {\it prime}, 
namely, not a power $w'^r$ of any other word $w'$ for any $r\geq2$, and 
is minimal in the cyclic rearrangements of $w$. We denote by $L_X$ the set of Lyndon words 
in $X^*$. 
It is well known that any nonempty word $w$ can be formed uniquely 
as a nonincreasing sequence of Lyndon words. 

\begin{thm}\label{LyndonFactorization}
For any nonempty word $w{\;\in\;}X^*$, there exists a unique nonincreasing sequence of 
Lyndon words $l_1,l_2,\cdots,l_r$ such that $w=l_1l_2{\cdots}l_r$. 
\end{thm}

\begin{proof}  For the proof, see for example \cite{Lothaire1997}. 
\end{proof}

Let us consider $R[X^*][[t]]$. 
Since $(1-lt)^{-1}=1+lt+(lt)^2+{\cdots}$ for every $l{\;\in\;}X^*$ 
in $R[X^*][[t]]$, Theorem \ref{LyndonFactorization} implies: 
\begin{equation}\label{SumOfWords1}
\prod_{l{\in}L_X}^{>}(1-lt^{|l|})^{-1}=\sum_{w{\in}X^*}wt^{|w|},
\end{equation}
in $R[X^*][[t]]$, 
where $\displaystyle \prod_{l{\in}L_X}^{>}$ means that the factors are multiplied 
in decreasing order. 
On the other hand, it follows that 
\begin{equation}\label{SumOfWords2}
\sum_{w{\in}X^*}wt^{|w|}=\{1-(x_1+{\cdots}+x_N)t\}^{-1}.
\end{equation}
(\ref{SumOfWords1}) and (\ref{SumOfWords2}) imply the following equation: 
\begin{equation}\label{InveseOfAmitsurIdentity}
\{1-(x_1+{\cdots}+x_N)t\}^{-1}=\prod_{l{\in}L_X}^{>}(1-lt^{|l|})^{-1}.
\end{equation}
From (\ref{InveseOfAmitsurIdentity}), we obtain: 
\begin{pro}
\begin{equation}\label{AmitsurIdentity}
1-(x_1+{\cdots}+x_N)t=\prod_{l{\in}L_X}^{<}(1-lt^{|l|}),
\end{equation}
where $\displaystyle \prod_{l{\in}L_X}^{<}$ means that the factors are multiplied 
in increasing order. 
\end{pro}
\begin{proof}  In order to show that 
\begin{equation}\label{ProdEquals1}
\Big{\{}\prod_{l{\in}L_X}^{>}(1-lt^{|l|})^{-1}\Big{\}}
\Big{\{}\prod_{l{\in}L_X}^{<}(1-lt^{|l|})\Big{\}}=1,
\end{equation}
we check that the coefficient of $t^r$ of the left hand side is equal to 
$1$ if $r=0$ and $0$ if $r>0$. Since 
$\prod_{l{\in}L_X}^{>}(1-lt^{|l|})^{-1}=
\prod_{l{\in}L_X}^{>}(1+lt^{|l|}+l^2t^{2|l|}+{\cdots})$, 
the coefficient of at most $r$-th power of $t$ in the left hand side of 
(\ref{ProdEquals1}) is the same as that of the product: 
\[
\Big{\{}\prod_{\substack{l{\in}L_X\\|l|{\leq}r}}^{>}
(1+lt^{|l|}+l^2t^{2|l|}+{\cdots})\Big{\}}
\Big{\{}\prod_{\substack{l{\in}L_X\\|l|{\leq}r}}^{<}(1-lt^{|l|})\Big{\}},
\]
for every nonnegative integer $r{\;\geq\;}0$. 
This is a finite product since $|X|<\infty$ and therefore 
is equal to $1$. Thus (\ref{ProdEquals1}) holds. 
\end{proof}

Let $[n]=\{1,2,\cdots,n\}$ with the natural order and $[n]{\times}[n]$ the 
cartesian product of $[n]$ with the lexicographic order derived from 
the natural order on $[n]$. 
We set $X=\{x(i,j)\;|\;(i,j){\;\in\;}[n]{\times}[n]\}$ with the total order derived from 
$[n]{\times}[n]$. 
For each matrix ${\bf A}=(a_{ij}){\;\in\;}\Mat(n,A)$, we define $\rho^{\bf A}$ to be the 
$R$-algebra homomorphism from $R[X^*]$ 
to $\Mat(n,A)$ defined by $\rho^{\bf A}(x(i,j))=a_{ij}{\bf E}_{ij}$, where ${\bf E}_{ij}$ denotes the 
$(i,j)$-matrix unit. Let ${\bf A}(i,j)=a_{ij}{\bf E}_{ij}$. 
One  can extend $\rho^{\bf A}$ to the $R$-algebra homomorphism $\rho^{\bf A}_t$ from $R[X^*][[t]]$ 
to $\Mat(n,A)[[t]]$ by defining $\rho^{\bf A}_t(t)=t$. 
For each Lyndon word $l=(i_1,j_1)(i_2,j_2)$ 
${\cdots}(i_r,j_r){\;\in\;}L_{[n]{\times}[n]}$, 
we put ${\bf A}_l={\bf A}(i_1,j_1){\bf A}(i_2,j_2){\cdots}{\bf A}(i_r,j_r)$. 
Then we have 
\begin{equation}\label{AmitsurForMatrix}
{\bf I}_n-\{{\bf A}(1,1)+{\bf A}(1,2)+{\cdots}+{\bf A}(n,n)\}t
=\prod_{l{\in}L_{[n]{\times}[n]}}^{<}({\bf I}_n-{\bf A}_lt^{|l|}).
\end{equation}
However we easily see that
\begin{equation*}
{\bf A}_l=\begin{cases}
a_{i_1i_2}a_{i_2i_3}{\cdots}a_{i_{r-1}i_r}a_{i_rj_r}{\bf E}_{i_1j_r} & 
\text{if $j_k=i_{k+1}$ for $k=1,{\cdots},r-1$}, \\
{\bf O}_n\;(\text{zero matrix of size $n$}) & \text{otherwise,}\end{cases}
\end{equation*}
and that ${\bf A}(1,1)+{\bf A}(1,2)+{\cdots}+{\bf A}(n,n)={\bf A}$, hence we obtain: 

\begin{pro}\label{AmitsurForMatrixThm}
Let ${\bf A}{\;\in\;}\Mat(n,A)$. Then, 
\begin{equation}\label{EqnAmitsurIdentity}
{\bf I}_n-{\bf A}t=\prod_{\substack{(i_1,j_1){\cdots}(i_r,j_r){\in}L_{[n]{\times}[n]}\\
j_k=i_{k+1}\;(k=1,{\cdots},r-1)}}^{<}({\bf I}_n-a_{i_1i_2}a_{i_2i_3}{\cdots}
a_{i_{r-1}i_r}a_{i_rj_r}{\bf E}_{i_1j_r}t^{r}),
\end{equation}
in $\Mat(n,A)[[t]]$. 
\end{pro}

\section{Exponentials and Logarithms of noncommutative formal power series}
Let $K=\QM[[t]]$ and $X=\{x_1,{\cdots},x_N\}$ as in the previous section. 
Then $K[X^*]=\QM[X^*][[t]]$, the ring of formal power series with 
coefficients in $\QM[X^*]$ where $t$ is considered as a central element 
in $K[X^*]$. Every $\alpha{\;\in\;}\QM[X^*][[t]]$ can be written as 
$\alpha=\sum_{k{\geq}0}\alpha_kt^k$ with $\alpha_k{\;\in\;}\QM[X^*]$. 
For an $\alpha$ which has zero constant term, 
we define the {\it exponential} of $\alpha$ to be the 
element of $\QM[X^*][[t]]$ as follows: 
\begin{equation}\label{DefExp}
\exp \alpha = \sum_{m{\geq}0}\dfrac{1}{m!}\alpha^m.
\end{equation}
Similarly, for an $\alpha$ whose constant term equals $1$, 
we define the {\it logarithm} of $\alpha$ as follows: 
\begin{equation}\label{DefLog}
\log \alpha = \sum_{n{\geq}1}\dfrac{(-1)^{n-1}}{n}(\alpha-1)^n.
\end{equation}
If $\alpha_0=1$, then $\alpha=1+\beta$ for some $\beta{\;\in\;}\QM[X^*][[t]]$ 
with zero constant term. Then we may write 
$\log \alpha=\log (1+\beta)=\sum_{n{\geq}1}(-1)^{n-1}n^{-1}\beta^n$ and 
$\log \alpha$ has zero constant term. 
Conversely, if $\beta$ has zero constant term, then the constant term of 
$\exp \beta$ equals $1$. 

\begin{thm}\label{ThmExpAndLog}
Let $\alpha,\,\beta{\;\in\;}\QM[X^*][[t]]$. 
\begin{enumerate}
\renewcommand{\labelenumi}{\rm (\arabic{enumi})}
\item
If $\alpha_0=1$, then $\exp\log \alpha =\alpha$. 
\item
If $\beta$ has zero constant term, then $\log\exp \beta=\beta$. 
\end{enumerate}
\end{thm}
\begin{proof}  
We shall prove (1). 
Letting $\alpha=1+\beta$, we have: 
\begin{equation}\label{EqnExpLog}
\exp\log \alpha=\exp\log (1+\beta)=
\sum_{m{\geq}0}\dfrac{1}{m!}\Big{(}\sum_{n{\geq}1}\dfrac{(-1)^{n-1}}
{n}\beta^n\Big{)}^m. 
\end{equation}
We consider the partial sum of (\ref{EqnExpLog}) as follows: 
\begin{equation}\label{EqnTruncatedExpLog}
\sum_{m=0}^k\dfrac{1}{m!}\Big{(}\sum_{n=1}^\ell\dfrac{(-1)^{n-1}}
{n}\beta^n\Big{)}^m.
\end{equation}
If $k,\ell{\;\geq\;}r$, then the coefficient of at most $r$-th power of 
$\beta$ in (\ref{EqnExpLog}) is equal to that in (\ref{EqnTruncatedExpLog}). 
On the other hand, by the Taylor expansion we have: 
\begin{equation}\label{EqnRealExpLog}
1+x=\exp\log (1+x)=\sum_{m{\geq}0}\dfrac{1}{m!}\Big{(}\sum_{n{\geq}1}\dfrac{(-1)^{n-1}}
{n}x^n\Big{)}^m, 
\end{equation}
for a real number $x$ with $|x|<1$. Similar to (\ref{EqnTruncatedExpLog}), 
if $k,\ell{\;\geq\;}r$, then the coefficient of at most $r$-th power of 
$x$ in (\ref{EqnRealExpLog}) is equal to that in (\ref{EqnRealTruncatedExpLog}): 
\begin{equation}\label{EqnRealTruncatedExpLog}
\sum_{m=0}^k\dfrac{1}{m!}\Big{(}\sum_{n=1}^\ell\dfrac{(-1)^{n-1}}{n}x^n\Big{)}^m.
\end{equation}
Hence the coefficient of at most $r$-th power of $\beta$ in (\ref{EqnExpLog}) 
is equal to that of $x$ in (\ref{EqnRealExpLog}). 
Since $r$ is an arbitrary nonnegative integer, the coefficient of any power of $\beta$ in (\ref{EqnExpLog}) 
coincide with the coefficient of the same power of $x$ in (\ref{EqnRealExpLog}). 
Moreover $\beta=\sum_{n{\geq}1}\alpha_nt^n$ implies that the coefficient of at most 
$r$-th power of $t$ is determined by the terms of at most $r$-th power of $\beta$ 
in (\ref{EqnExpLog}). 
Since $r$ is arbitrary, we obtain that $\exp\log (1+\beta)=1+\beta$ in $\QM[X^*][[t]]$. 
(2) can be proven in the same manner. 
\end{proof}

\begin{pro}\label{ProExp}
Assume $\alpha,\,\beta{\;\in\;}\QM[X^*][[t]]$ have zero constant terms. 
If $\alpha\beta=\beta\alpha$, then 
$\exp(\alpha+\beta)=(\exp \alpha)(\exp \beta)$. 
\end{pro}
\begin{proof}  
By easy calculations, we have: 
\[
\exp(\alpha+\beta)
=\sum_{m{\geq}0}\sum_{\substack{k,\ell{\geq}0\\k+\ell=m}}\dfrac{1}{k!\ell!}\alpha^k
\beta^{\ell}=(\exp \alpha)(\exp \beta).
\]
\end{proof}

\begin{pro}\label{ProLog}
Assume $\alpha,\,\beta{\;\in\;}\QM[X^*][[t]]$ satisfy $\alpha_0=\beta_0=1$ and put 
$\alpha'=\alpha-1,\ \beta'=\beta-1$. 
If $\alpha'\beta'=\beta'\alpha'$, then 
$\log(\alpha\beta)=\log \alpha+\log \beta$. 
\end{pro}
\begin{proof}  
Let $p=\log(1+\alpha'),\ q=\log(1+\beta')$. Both $p$ and $q$ have zero constant term. 
By Theorem \ref{ThmExpAndLog} (1), 
$\exp p=1+\alpha',\,\exp q=1+\beta'$. 
Since $pq=qp$, $(1+\alpha')(1+\beta')=(\exp p)(\exp q)=\exp(p+q)$ by Proposition 
\ref{ProExp}. Thus 
$\log((1+\alpha')(1+\beta'))=\log\exp(p+q)=p+q=\log(1+\alpha')+\log(1+\beta')$ 
by Theorem \ref{ThmExpAndLog} (2). 
\end{proof}

\begin{cor}\label{CorLog}
Assume $\alpha{\;\in\;}\QM[X^*][[t]]$ satisfies $\alpha_0=1$ and put 
$\alpha'=1-\alpha$. Then $\log\alpha^{-1}=-\log\alpha$ where 
$\alpha^{-1}=(1-\alpha')^{-1}=1+\alpha'+(\alpha')^2+{\cdots}$ is in 
$\QM[X^*][[t]]$. 
\end{cor}
\begin{proof}  It follows from Proposition \ref{ProLog} that: 
\begin{equation*}
\begin{split}
0&=\log 1 = \log((1-\alpha')(1-\alpha')^{-1})\\
&=\log((1-\alpha')(1+\alpha'+(\alpha')^2+{\cdots}))\\
&=\log(1-\alpha')+\log(1+\alpha'+(\alpha')^2+{\cdots})\\
&=\log(1-\alpha')+\log((1-\alpha')^{-1}). 
\end{split}
\end{equation*}
\end{proof}

Let $A$ be a $\QM$-algebra. 
We set $X=\{x(i,j)\;|\;(i,j){\;\in\;}[n]{\times}[n]\}$ in the same manner as in 
the Section $2$. 
For each matrix ${\bf A}=(a_{ij}){\;\in\;}\Mat(n,A)$, we define $\sigma^{\bf A}$ to be the 
$\QM$-algebra homomorphism from $\QM[X^*]$ 
to $A$ defined by $\sigma^{\bf A}(x(i,j))=a_{ij}$. 
It is clear that Theorem \ref{ThmExpAndLog}, Proposition \ref{ProExp}, 
Proposition \ref{ProLog} and Corollary \ref{CorLog} hold even if one replaces 
$\QM[X^*]$ with $A$ since 
one can extend $\sigma^{\bf A}$ to the $\QM$-algebra homomorphism $\sigma^{\bf A}_t$ 
from $\QM[X^*][[t]]$ to $A[[t]]$ by defining $\sigma^{\bf A}_t(t)=t$.

\section{The Study determinant of a quaternionic matrix}
Let $\HM$ be the set of quaternions. $\HM$ is a noncommutative associative 
algebra over $\RM$, whose underlying real vector space has dimension $4$ 
with a basis $1,i,j,k$ which satisfy the following relations: 
\[
i^2=j^2=k^2=-1,\quad ij=-ji=k,\quad jk=-kj=i,\quad ki=-ik=j.
\]
For $x=x_0+x_1i+x_2j+x_3k{\;\in\;}\HM$, 
$x^*$ denotes the conjugate of $x$ in $\HM$ 
which is defined by $x^*=x_0-x_1i-x_2j-x_3k$, 
and $\RP x=x_0$ the real part of $x$. 
One can easily check $xx^*=x^*x$, $(x^*)^n=(x^n)^*$, and 
$x^{-1}=x^*/|x|^2$ for $x{\;\neq\;}0$. 
Hence, $\HM$ constitutes a skew field. 
We call $|x|=\sqrt{xx^*}=\sqrt{x^*x}=\sqrt{x_0^2+x_1^2+x_2^2+x_3^2}$ the norm of $x$. 
Indeed, $|{\,\cdot\,}|$ satisfies 
\begin{enumerate}
\renewcommand{\labelenumi}{\rm (\arabic{enumi})}
\setlength{\itemsep}{-1em}
\item $|x|{\;\geq\;}0$, and moreover $|x|=0 \Leftrightarrow x=0$,\\
\item $|xy|=|x||y|$,\\
\item $|x+y|{\;\leq\;}|x|+|y|$.
\end{enumerate}
Any quaternion $x$ can be presented by two complex numbers 
$x=a+jb$ uniquely. Such a presentation is called the 
{\it symplectic decomposition}. 
Two complex numbers $a$ and $b$ are called the {\it simplex part} and the 
{\it perplex part} 
of $x$ respectively. 
The symplectic decomposition is also valid for a quaternionic matrix, namely 
a matrix whose entries are quaternions. 
$\Mat(m,n,\HM)$ denotes the set of $m{\times}n$ quaternionic matrices and 
$\Mat(m,\HM)$ the set of $m{\times}m$ quaternionic matrices. 
For ${\bf M}{\;\in\;}\Mat(m,n,\HM)$, we can write ${\bf M}={\bf M}^S+j{\bf M}^P$ 
uniquely where ${\bf M}^S,{\bf M}^P{\;\in\;}\Mat(m,n,\CM)$. 
${\bf M}^S$ and ${\bf M}^P$ are called the {\it simplex part} and the 
{\it perplex part} of ${\bf M}$ respectively. 
We define $\psi$ to be the map from $\Mat(m{\times}n,\mathbb{H})$ to 
$\Mat(2m{\times}2n,\mathbb{C})$ as follows: 
\[
\psi : \Mat(m{\times}n,\mathbb{H}){\;\longrightarrow\;}\Mat(2m{\times}2n,\mathbb{C})
\quad{\bf M}{\;\mapsto\;}\begin{bmatrix}{\bf M}^S&-\overline{{\bf M}^P}\\{\bf M}^P&
\overline{{\bf M}^S}\end{bmatrix},
\]
where $\overline{\bf A}$ is the complex conjugate of a complex matrix ${\bf A}$. 
Then $\psi$ is an $\RM$-linear map. We also have: 

\begin{lem}\label{PsiLem}
Let ${\bf M}{\;\in\;}\Mat(m{\times}n,\HM)$ and ${\bf N}{\;\in\;}
\Mat(n{\times}m,\HM)$. Then 
\[ \psi({\bf M}{\bf N})=\psi({\bf M})\psi({\bf N}). \]
\end{lem}

\begin{proof} 
Let ${\bf M}={\bf A}+j{\bf B}$ and ${\bf N}={\bf C}+j{\bf D}$ 
be the symplectic decompositions of ${\bf M}$ and ${\bf N}$. Then, 
\[ {\bf M}{\bf N}=({\bf A}+j{\bf B})({\bf C}+j{\bf D})
={\bf A}{\bf C}+{\bf A}j{\bf D}+j{\bf B}{\bf C}+j{\bf B}j{\bf D}. \]
Since  
${\bf X}j=j\overline{\bf X}$ for every complex matrix ${\bf X}$, 
we obtain: 
\[ {\bf M}{\bf N}={\bf A}{\bf C}-\overline{{\bf B}}{\bf D}
+j(\overline{{\bf A}}{\bf D}+{\bf B}{\bf C}), \]
and therefore
\[ \psi({\bf M}{\bf N})
=\begin{bmatrix}{\bf A}{\bf C}-\overline{{\bf B}}{\bf D}
&-{\bf A}\overline{{\bf D}}-\overline{{\bf B}}\overline{{\bf C}}\\
\overline{{\bf A}}{\bf D}+{\bf B}{\bf C}
&\overline{{\bf A}}\overline{{\bf C}}-{\bf B}\overline{{\bf D}}
\end{bmatrix}.
\]
On the other hand, 
\[ \psi({\bf M})\psi({\bf N})=
\begin{bmatrix}{\bf A}&-\overline{{\bf B}}\\
{\bf B}&\overline{{\bf A}}
\end{bmatrix}
\begin{bmatrix}{\bf C}&-\overline{{\bf D}}\\
{\bf D}&\overline{{\bf C}}
\end{bmatrix}
=\begin{bmatrix}{\bf A}{\bf C}-\overline{{\bf B}}{\bf D}
&-{\bf A}\overline{{\bf D}}-\overline{{\bf B}}\overline{{\bf C}}\\
{\bf B}{\bf C}+\overline{{\bf A}}{\bf D}
&-{\bf B}\overline{{\bf D}}+\overline{{\bf A}}\overline{{\bf C}}
\end{bmatrix}.
\]
Thus $\psi({\bf M}{\bf N})=\psi({\bf M})\psi({\bf N})$ holds. 
\end{proof}

\begin{pro}
If $m=n$, then $\psi$ is an injective $\mathbb{R}$-algebra homomorphism. 
\end{pro}

\begin{proof} 
By Lemma \ref{PsiLem}, $\psi$ is an $\RM$-algebra homomorphism. 
Injectivity of $\psi$ is clear. 
\end{proof}

As in \cite{Aslaksen1996}, 
one can characterize the image of $\psi$ in $\Mat(2n,\CM)$ as follows. 
\begin{lem}\label{ImageOfPsi}
Let ${\bf J}$ be the $2n{\times}2n$ matrix defined as follows: 
\[
{\bf J}=\begin{bmatrix}
{\bf 0} & -{\bf I}_n \\ {\bf I}_n & {\bf 0}
\end{bmatrix}.
\]
Then, 
\[
\psi(\Mat(n,\HM))=\{{\bf N}{\;\in\;}\Mat(2n,\CM)\,|\, 
{\bf J}{\bf N}=\overline{\bf N}{\bf J}\},
\]
where $\overline{\bf N}$ denotes the complex conjugate of ${\bf N}$. 
\end{lem}

\begin{proof} 
For the proof, see \cite{Aslaksen1996}. 
\end{proof}

In \cite{Study1920}, Study defined a determinant of a $n{\times}n$ 
quaternionic matrix which we denote by $\Sdet$ as follows: 
\[
\Sdet({\bf M})=\det(\psi({\bf M})),
\]
where $\det$ is the ordinary determinant. 
$\Sdet$ is called the {\em Study determinant}. 

$\Sdet$ can be extended to the one for matrices whose entries are formal power series 
with coefficients in $\HM$ in the following way. 
Let $\HM[[t]]$ be the ring of formal power series with 
coefficients in $\HM$. 
We notice that $t$ is a commuting indeterminate for $\HM$, that is, 
$th=ht$ for any $h{\;\in\;}\HM$. 
For $\alpha=\sum_{k{\geq}0}\alpha_kt^k{\;\in\;}\HM[[t]]$, the {\it conjugate} of $\alpha$ 
is defined by $\alpha^*=\sum_{k{\geq}0}\alpha_k^*t^k$. 
Furthermore, let $\alpha_k=\alpha_k^{S}+j\alpha_k^{P}$ be the symplectic decomposition 
of $\alpha_k$. Then the symplectic decomposition of $\alpha$ is defined by 
\begin{equation*}
\alpha=\alpha^{S}+j\alpha^{P}=\sum_{k{\geq}0}\alpha_k^{S}t^k+j\sum_{k{\geq}0}\alpha_k^{P}t^k.
\end{equation*}
One can extend $\psi$ to the injective homomorphism $\psi_t$ of $\RM$-algebras from 
$\Mat(n,\HM)[[t]]$ to $\Mat(2n,\CM)[[t]]$ by defining $\psi_t(t)=t$. 
On the other hand, since $\Mat(2n,\CM)[[t]] = \Mat(2n,\CM[[t]])$, 
one can also extend the determinant $\det : \Mat(2n,\CM){\longrightarrow}\CM$ 
to $\det_t : \Mat(2n,\CM)[[t]]{\longrightarrow}\CM[[t]]$ in the natural manner 
since $\det$ is a polynomial of entries of a matrix.  
Both $\psi_t$ and $\det_t$ are continuous, thus 
\[
{\det}_t{\cdot}\psi_t : \Mat(n,\HM)[[t]]{\longrightarrow}\CM[[t]]
\]
is also continuous with respect to the topology of formal power series. 
We call $\det_t{\cdot}\psi_t$ the {\it Study determinant} for $\Mat(n,\HM[[t]])$ 
and denote by $\Sdet_t$. 
We notice that the restriction of $\Sdet_t$ to $\Mat(n,\HM)$ yields $\Sdet$.

Before stating properties of $\Sdet_t$, we mention some useful facts. 
\begin{lem}\label{DetPartitionedMat}
If ${\bf A}, {\bf B}, {\bf C}, {\bf D}$ are complex square matrices with same size. 
Suppose that ${\bf A}$ is invertible and ${\bf A}{\bf C}={\bf C}{\bf A}$, then 
\[
\det  
\left[ 
\begin{array}{cc}
{\bf A} & {\bf B} \\
{\bf C} & {\bf D}   
\end{array} 
\right] 
= \det ({\bf A}{\bf D}-{\bf C}{\bf B}) . 
\]
\end{lem}

\begin{proof} 
For the proof, see for example \cite{Zhang2011}. 
\end{proof}

\begin{rem}\label{ExtDetPartitionedMat}
We notice that the Lemma \ref{DetPartitionedMat} can be extended to matrices whose entries 
belong to $\CM[[t]]$ immediately. 
\end{rem}

\begin{lem}\label{ImageOfPsit}
\[
\psi_t(\Mat(n,\HM[[t]]))=\{{\bf N}{\;\in\;}\Mat(2n,\CM[[t]])\,|\, 
{\bf J}{\bf N}=\overline{\bf N}{\bf J}\}.
\]
\end{lem}

\begin{proof}  By Lemma \ref{ImageOfPsi}, it follows that
\begin{equation*}
\begin{split}
&\exists {\bf M}=\sum_{k{\geq}0}{\bf M}_kt^k{\;\in\;}\Mat(n,\HM[[t]]),\;
\psi_t({\bf M})={\bf N}=\sum_{k{\geq}0}{\bf N}_kt^k\\
&{\Leftrightarrow\ } \psi({\bf M}_k)={\bf N}_k \quad \text{for ${\forall}k{\geq}0$}
{\ \Leftrightarrow\ } {\bf J}{\bf N}_k=\overline{\bf N}_k{\bf J} \quad 
\text{for ${\forall}k{\geq}0$}\\
&{\Leftrightarrow\ }{\bf J}{\bf N}=\overline{\bf N}{\bf J}.
\end{split}
\end{equation*}
Hence, the assertion holds. 
\end{proof}

\begin{lem}\label{ProdOfFAndConjF}
Let $\alpha=\sum_{k{\geq}0}\alpha_kt^k{\;\in\;}\HM[[t]]$, and 
$\alpha_k=\alpha_k^{S}+j\alpha_k^{P}$ 
the symplectic decomposition. Then, 
\[
\alpha\alpha^*=\alpha^*\alpha=
\sum_{m{\geq}0}\sum_{k+\ell=m}
(\alpha_k^{S}\overline{\alpha_{\ell}^{S}}+\alpha_k^{P}\overline{\alpha_{\ell}^{P}})t^m 
=\alpha^{S}\overline{\alpha^{S}}+\alpha^{P}\overline{\alpha^{P}}{\;\in\;}\CM[[t]].
\]
\end{lem}

\begin{proof}  By a direct calculation, we have: 
\begin{equation*}
\begin{split}
\alpha\alpha^*&=(\sum_{k{\geq}0}(\alpha_k^{S}+j\alpha_k^{P})t^k)
(\sum_{\ell{\geq}0}({\alpha_{\ell}^{S}}+j\alpha_{\ell}^{P})^*t^{\ell})\\
&=(\sum_{k{\geq}0}(\alpha_k^{S}+j\alpha_k^{P})t^k)
(\sum_{\ell{\geq}0}(\overline{\alpha_{\ell}^{S}}-\overline{\alpha_{\ell}^{P}}j)t^{\ell})\\
&=\sum_{m{\geq}0}\sum_{k+\ell=m}
\Big{\{}(\alpha_k^{S}\overline{\alpha_{\ell}^{S}}+\overline{\alpha_k^{P}}\alpha_{\ell}^{P})
-j(\overline{\alpha_k^{S}}\alpha_{\ell}^{P}-\alpha_k^{P}\overline{\alpha_{\ell}^{S}})
\Big{\}}t^m\\
&=\sum_{m{\geq}0}\sum_{k+\ell=m}
(\alpha_k^{S}\overline{\alpha_{\ell}^{S}}+\overline{\alpha_k^{P}}\alpha_{\ell}^{P})t^m\\
&=\alpha^{S}\overline{\alpha^{S}}+\alpha^{P}\overline{\alpha^{P}}. 
\end{split}
\end{equation*}
Furthermore, $\alpha^{S}\overline{\alpha^{S}}+\alpha^{P}\overline{\alpha^{P}}$ belongs to 
$\CM[[t]]$. 
Similarly, the same conclusion holds for $\alpha^*\alpha$. 
\end{proof}

Lemma \ref{ProdOfFAndConjF} immediately implies: 
\begin{cor}\label{CommutativityOfAlpahConjAlpha}
For all $\alpha,\beta{\;\in\;}\HM[[t]]$, 
$\alpha\alpha^*\beta\beta^*=\beta\beta^*\alpha\alpha^*$. 
\end{cor}

\begin{pro}\label{SdettProperties}
\begin{enumerate}
\renewcommand{\labelenumi}{\rm (\roman{enumi})}
\item
$\Sdet_t({\bf M}){\;\in\;}\RM[[t]]$ 
for ${\bf M}{\;\in\;}\Mat(n,\HM[[t]])$

\item
The constant term of $\Sdet_t({\bf M})$ is equal to 
$0{\ \Leftrightarrow\ }{\bf M} \text{ is singular.}$

\item
$\Sdet_t({\bf MN})=\Sdet_t({\bf M})\Sdet_t({\bf N})$ for 
${\bf M},{\bf N}{\;\in\;}\Mat(n,\HM[[t]])$.

\item
If ${\bf N}$ is obtained from ${\bf M}$ by adding a left-multiple of a row 
to another row or a right-multiple of a column to another column, 
then $\Sdet_t({\bf N})=\Sdet_t({\bf M})$. 

\item
If ${\bf N}$ is obtained from ${\bf M}$ by exchanging a row 
with another row or a column with another column, 
then $\Sdet_t({\bf N})=\Sdet_t({\bf M})$. 

\item
$\Sdet_t(\alpha{\bf M})=\Sdet_t({\bf M}\alpha)=(\alpha\alpha^*)^{n}
\Sdet_t({\bf M})$ 
for ${\bf M}{\;\in\;}\Mat(n,\HM[[t]])$ and 
$\alpha{\;\in\;}\HM[[t]]$. 

\item
If ${\bf M}{\;\in\;}\Mat(n,\HM[[t]])$ is of the form:
\[
{\bf M}=
\begin{bmatrix}
\lambda_1 & * &{\cdots}& * \\
0 & \lambda_2 & & * \\
{\vdots}& & \ddots &{\vdots}\\
0 & 0 &{\cdots}&\lambda_n
\end{bmatrix} \quad\text{or}\quad 
\begin{bmatrix}
\lambda_1 & 0 &{\cdots}& 0 \\
* & \lambda_2 & & 0 \\
{\vdots}& & \ddots &{\vdots}\\
* & * &{\cdots}&\lambda_n
\end{bmatrix},
\]
Then, $\Sdet_t({\bf M})=\prod_{i=1}^n\lambda_i\lambda_i^*$.

\item  Let ${\bf A}$ be an $m \times n$ matrix and ${\bf B}$ an $n \times m$ matrix.
Then 
\[
\Sdet_t ({\bf I}_m -{\bf A}{\bf B})= \Sdet_t ({\bf I}_n -{\bf B}{\bf A}) . 
\]
\end{enumerate}
\end{pro}

\begin{proof}  
(i)\ \ It follows from Lemma \ref{ImageOfPsit} that 
\[
{\det}_t({\bf N})={\det}_t({\bf J}^{-1}\overline{\bf N}{\bf J})
={\det}_t(\overline{\bf N})=\overline{{\det}_t({\bf N})},
\]
for any ${\bf N}{\;\in\;}\psi_t(\Mat(n,\HM[[t]]))$. 
Thus $\Sdet_t({\bf M})$ must belongs to $\RM[[t]]$. \\
(ii)\ \ It is known that a matrix over a commutative ring is invertible if and 
only if its determinant is invertible, and that 
a formal power series is invertible if and only if 
its constant term is invertible. 
Hence, for any ${\bf M}{\;\in\;}\Mat(n,\HM[[t]])$, 
the constant term of $\det_t{\cdot}\psi_t({\bf M})$ is equal to $0 {\ \Leftrightarrow\ }$
$\det_t{\cdot}\psi_t({\bf M})$ is not invertible ${\ \Leftrightarrow\ }$ 
$\psi_t({\bf M})$ is singular. Taking conjugate and inverse of 
$\overline{\psi_t({\bf M})}{\bf J}={\bf J}{\psi_t({\bf M})}$, we have: 
\[
{\bf J}\psi_t({\bf M})^{-1}=\overline{\psi_t({\bf M})^{-1}}{\bf J}.
\]
Hence by Lamma \ref{ImageOfPsit}, $\psi_t({\bf M})^{-1}{\;\in\;}\psi_t(\Mat(n,\HM[[t]]))$ 
and it follows that $\psi_t({\bf M})$ is singular $\ \Leftrightarrow\ {\bf M}$ is singular. \\
(iii)\ \ Since $\psi_t$ is an algebra homomorphism, (iii) immediately holds. \\
(iv)\ \ Let $b=b_1+jb_2{\;\in\;}\HM[[t]]$ ($b_1,b_2{\;\in\;}\CM[[t]]$). 
It suffice to show that for $k{\;\neq\;}\ell$, the $n{\times}n$ matrix 
${\bf B}_{k\ell}(b)={\bf I}_n+b{\bf E}_{k\ell}$ satisfies $\Sdet_t({\bf B}_{k\ell}(b))=1$. 
Since four blocks of 
\[
\psi_t({\bf B}_{k\ell}(b))=\begin{bmatrix}
{\bf I}_n+b_1{\bf E}_{k\ell} & -\overline{b_2}{\bf E}_{k\ell} \\
b_2{\bf E}_{k\ell} & {\bf I}_n+\overline{b_1}{\bf E}_{k\ell}
\end{bmatrix}
\]
commute with each other, we can apply Remark \ref{ExtDetPartitionedMat} and 
obtain: 
\[
{\det}_t{\cdot}\psi_t({\bf B}_{k\ell}(b))
={\det}_t(({\bf I}_n+b_1{\bf E}_{k\ell})({\bf I}_n+\overline{b_1}{\bf E}_{k\ell})
+b_2{\bf E}_{k\ell}\overline{b_2}{\bf E}_{k\ell})=1.
\]
\noindent
(v)\ \ We prove for row exchange. 
Putting ${\bf N}={\bf P}{\bf M}$ for the permutation matrix ${\bf P}$ 
corresponding to some transposition, it follows that: 
\begin{equation*}
\begin{split}
\Sdet_t({\bf N})&=\Sdet_t({\bf P}{\bf M})={\det}_t
\Big{(}\begin{bmatrix}{\bf P} & {\bf O} \\ {\bf O} & {\bf P}\end{bmatrix}
\psi_t({\bf M})\Big{)}\\
&={\det}_t
\Big{(}\begin{bmatrix}{\bf P} & {\bf O} \\ {\bf O} & {\bf P}\end{bmatrix}\Big{)}
{\det}_t\Big{(}\psi_t({\bf M})\Big{)}=\det({\bf P})^2\Sdet_t({\bf M})\\
&=\Sdet_t({\bf M}).
\end{split}
\end{equation*}
\noindent
(vi)\ \ Let ${\bf M} \in \Mat(n,\HM[[t]])$ and 
$\alpha = \alpha {}^S +j \alpha^P \in {\HM[[t]]} \ 
( \alpha^S , \alpha^P \in {\CM[[t]]})$ be the symplectic decomposition. 
Then by using Lemma \ref{PsiLem}, Remark \ref{ExtDetPartitionedMat} and 
Lemma \ref{ProdOfFAndConjF}, we have: 
\[
\begin{array}{rcl}
{\Sdet_t} ( \alpha {\bf M}) & = & \det_t ( \psi_t ( \alpha {\bf M}))
= \det_t ( \psi_t ( \alpha {\bf I}_n ) \psi_t ({\bf M})) \\
\  &   &                \\ 
\  & = & \det_t  
\left[ 
\begin{array}{cc}
\alpha^S {\bf I}_n & - \overline{ \alpha^P } {\bf I}_n \\ 
\alpha^P {\bf I}_n & \overline{ \alpha^S } {\bf I}_n  
\end{array} 
\right]
\det_t (\psi_t ({\bf M})) \\ 
\  &   &                \\ 
\  & = & 
\det_t (( \alpha^S {\bf I}_n )( \overline{\alpha^S} {\bf I}_n )
+( \alpha^P {\bf I}_n )( \overline{\alpha^P} {\bf I}_n )) 
{\Sdet_t} ({\bf M}) \\ 
\  &   &                \\ 
\  & = & 
\det_t ((\alpha^S\overline{\alpha^S} +\alpha^P\overline{\alpha^P}) {\bf I}_n ) {\Sdet_t} ({\bf M}) \\ 
\  &   &                \\ 
\  & = & (\alpha\alpha^*)^n {\Sdet_t} ({\bf M}) . 
\end{array}
\]
In the same way, we can deduce $\Sdet_t({\bf M}\alpha)=(\alpha\alpha^*)^n\Sdet_t({\bf M})$. 

\noindent
(vii)\ \ For a $2n{\times}2n$ matrix ${\bf N}$ and any two subsets 
$I=\{i_1,i_2,{\cdots},i_r\},\, J=\{j_1,j_2,{\cdots},j_s\}$ of $[2n]$, 
${\bf N}^{IJ}$ denotes the submatrix obtained from ${\bf N}$ by deleting 
$i_1,i_2,{\cdots},i_r$ th rows and $j_1,j_2,{\cdots},j_s$ th columns. 
Then by definitions of $\Sdet_t$ and $\psi_t$, we get the following: 
\begin{equation*}
\begin{split}
\Sdet_t({\bf M})&=
{\det}_t(\psi_t({\bf M}))
={\det}_t \begin{bmatrix}{\bf M}^S & -\overline{{\bf M}^P}\\
{\bf M}^P & \overline{{\bf M}^S}
\end{bmatrix}\\
&={\det}_t \begin{bmatrix}
\lambda_1^{S}& * & {\cdots} & * & -\overline{\lambda_1^{P}} & * & {\cdots} & * \\
0 &\lambda_2^{S}&  & * & 0 & -\overline{\lambda_2^{P}} & & * \\
{\vdots} &   & \ddots & {\vdots} & {\vdots} &   & \ddots & {\vdots} \\
0 & 0 & {\cdots} &\lambda_n^{S} & 0 & 0 & {\cdots} & -\overline{\lambda_n^{P}} \\
\lambda_1^{P} & * & {\cdots} & * & \overline{\lambda_1^{S}}& * & {\cdots} & * \\
0 & \lambda_2^{P} &  & * & 0 &\overline{\lambda_2^{S}}& & * \\
{\vdots} &   & \ddots & {\vdots} & {\vdots} &   & \ddots & {\vdots} \\
0 & 0 & {\cdots} & \lambda_n^{P} & 0 & 0 & {\cdots} &\overline{\lambda_n^{S}}
\end{bmatrix}\\
&=\lambda_1^{S}\psi_t({\bf M})^{\{1\}\{1\}}+(-1)^{n+2}\lambda_1^{P}\psi_t({\bf M})^{\{n+1\}\{1\}}\\
&=\lambda_1^{S}\overline{\lambda_1^{S}}\psi_t({\bf M})^{\{1,n+1\}\{1,n+1\}}
+\lambda_1^{P}\overline{\lambda_1^{P}}\psi_t({\bf M})^{\{1,n+1\}\{1,n+1\}}\\
&=(\lambda_1^{S}\overline{\lambda_1^{S}}+\lambda_1^{P}\overline{\lambda_1^{P}})
\psi_t({\bf M})^{\{1,n+1\}\{1,n+1\}}\\
&=\lambda_1\lambda_1^*\psi_t({\bf M})^{\{1,n+1\}\{1,n+1\}}\\
&=\lambda_1\lambda_1^*\lambda_2\lambda_2^*\psi_t({\bf M})^{\{1,2,n+1,n+2\}\{1,2,n+1,n+2\}}
={\cdots}\\
&=\prod_{i=1}^n\lambda_i\lambda_i^*
\end{split}
\end{equation*}

\noindent
(viii)\ \ Let ${\bf A}$ be an $m \times n$ matrix and ${\bf B}$ an $n \times m$ matrix.
Then, by the definition of the Study determinant, we have: 
\begin{equation*}
\begin{split}
{\Sdet_t}({\bf I}_m-{\bf A}{\bf B})&={\det}_t(\psi_t({\bf I}_m -{\bf A}{\bf B})) 
={\det}_t({\bf I}_{2m}-\psi_t({\bf A})\psi_t({\bf B})) \\
&= {\det}_t({\bf I}_{2n}-\psi_t({\bf B})\psi_t({\bf A})) 
= {\det}_t(\psi_t({\bf I}_n-{\bf B}{\bf A})) \\
&= {\Sdet_t}({\bf I}_n-{\bf B}{\bf A}). 
\end{split}
\end{equation*}
\end{proof}

\begin{rem}
$\Sdet_t$ is not multilinear as $\det_t$ is. Furthermore, 
$\Sdet_t({}^T\!{\bf M})=\Sdet_t({\bf M})$ does not hold in general where 
${}^T\!{\bf M}$ is the transpose of ${\bf M}$. 
\end{rem}

\section{The Ihara zeta function of a graph} 
Let $G=(V(G)$, $E(G))$ be a finite connected graph with the set $V(G)$ of 
vertices and the set $E(G)$ of unoriented edges $uv$ 
joining two vertices $u$ and $v$. 
We assume that $G$ has neither loops nor multiple edges throughout. 
For $uv \in E(G)$, an arc $(u,v)$ is the oriented edge from $u$ to $v$. 
Let $D(G)=\{\,(u,v),\,(v,u)\,\mid\,uv{\;\in\;}E(G)\}$ and 
$|V(G)|=n,\;|E(G)|=m,\;|D(G)|=2m$. 
For $e=(u,v){\;\in\;}D(G)$, $o(e)=u$ denotes the {\it origin} and $t(e)=v$ the {\it terminal} 
of $e$ respectively. 
Furthermore, let $e^{-1}=(v,u)$ be the {\em inverse} of $e=(u,v)$. 
The {\em degree} $\deg v = \deg {}_G \  v$ of a vertex $v$ of $G$ is the number of edges 
incident to $v$. 
For a natural number $k$, a graph $G$ is called {\em $k$-regular } if $\deg {}_G \  v=k$ 
for each vertex $v$ of $G$. 
A {\em path $P$ of length $\ell$} in $G$ is a sequence 
$P=(e_1, \cdots ,e_{\ell})$ of $\ell$ arcs such that $e_i \in D(G)$ and 
$t(e_i)=o(e_{i+1})$ for $i{\;\in\;}\{1,\cdots,\ell-1\}$. 
We set $o(P)=o(e_1)$ and $t(P)=t(e_{\ell})$. $|P|$ denotes the length of $P$. 
We say that a path $P=(e_1, \cdots ,e_{\ell})$ has a {\em backtracking} 
if $ e_{i+1} =e_i^{-1} $ for some $i(1{\leq}i{\leq}\ell-1)$.
A path $P$ is said to be a {\em cycle} if $t(P)=o(P)$. 
The {\em inverse} of a path 
$P=(e_1,\cdots,e_{\ell})$ is the path 
$(e_{\ell}^{-1},\cdots,e_1^{-1})$ and is denoted by $P^{-1}$.

An equivalence relation between cycles is given as follows. 
Two cycles $C_1 =(e_1, \cdots ,e_{\ell})$ and 
$C_2 =(f_1, \cdots ,f_{\ell})$ are said to be {\em equivalent} if there exists 
$k$ such that $f_j =e_{j+k}$ for all $j$ where indices are treated modulo $\ell$. 
Let $[C]$ be the equivalence class which contains the cycle $C$. 
Let $B^r$ be the cycle obtained by going $r$ times around a cycle $B$. 
Such a cycle is called a {\em power} of $B$. 
A cycle $C$ is said to be {\em reduced} if 
both $C$ and $C^2$ have no backtracking. 
Furthermore, a cycle $C$ is said to be {\em prime} if it is not a power of 
a strictly smaller cycle. 
Note that each equivalence class of prime, reduced cycles of a graph $G$ 
corresponds to a unique conjugacy class of 
the fundamental group $ \pi {}_1 (G,v)$ of $G$ at a vertex $v$ of $G$.

The {\em Ihara zeta function} of a graph $G$ is 
a function of $t \in {\bf C}$ with $|t|$ sufficiently small, 
defined by 
\[
{\bf Z} (G, t)= {\bf Z}_G (t)= \prod_{[C]} (1- t^{|C|} )^{-1} ,
\]
where $[C]$ runs over all equivalence classes of prime, reduced cycles 
of $G$. 

Let ${\bf B}=( {\bf B}_{ef} )_{e,f \in D(G)} $ and 
${\bf J}_0=( {\bf J}_{ef} )_{e,f \in D(G)}$ be $2m \times 2m$ matrices 
defined as follows: 
\[
{\bf B}_{ef} =\left\{
\begin{array}{ll}
1 & \mbox{if $t(e)=o(f)$, } \\
0 & \mbox{otherwise,}
\end{array}
\right.
{\bf J}_{ef} =\left\{
\begin{array}{ll}
1 & \mbox{if $f= e^{-1} $, } \\
0 & \mbox{otherwise.}
\end{array}
\right.
\]
Then the matrix ${\bf B} - {\bf J}_0 $ is called the {\em edge matrix} of $G$.

\begin{thm}[Hashimoto\cite{Hashimoto1989}; Bass\cite{Bass1992}]
Let $G$ be a finite connected graph. 
Then the reciprocal of the Ihara zeta function of $G$ is given by 
\[
{\bf Z} (G, t)^{-1} =\det ( {\bf I} -t ( {\bf B} - {\bf J}_0 ))
=(1- t^2 )^{r-1} \det ( {\bf I} -t {\bf A} + 
t^2 ({\bf D} -{\bf I} )), 
\]
where $r$ and ${\bf A}$ are the Betti number and the adjacency matrix 
of $G$, respectively, and ${\bf D} =( d_{ij} )$ is the diagonal matrix 
with $d_{ii} = \deg v_i $, $V(G)= \{ v_1 , \cdots , v_n \} $. 
${\bf Z} (G, t)$ has an expression as the exponential of a generating 
function as follows: 
\[
{\bf Z} (G, t)=\exp\Big{(}\sum_{k{\geq}1}\dfrac{N_k}{k}t^k\Big{)}, 
\]
where $N_k$ is the number of equivalence classes of reduced cycles of length $k$. 
\end{thm}

The matrix ${\bf D}$ is called the {\em degree matrix} of $G$. 

A weighted zeta function of a graph was first defined in Hashimoto \cite{Hashimoto1989} 
by giving a weight for each edge of a graph. 
After that, it was generalized in 
Stark and Terras \cite{ST1996} by giving a weight for each arc of a graph as follows. 
Let $D(G)= \{ e_1, \cdots, e_m, e_{m+1} $, $ \cdots , e_{2m} \}$ where 
$e_{m+i} = e^{-1}_i$ for $1 \leq i \leq m$. For each arc $e_i$ ($i=1,\cdots,2m$), 
associate a complex number $u_i$, and set ${\bf u} = (u_1 , \cdots , u_{2m} )$. 
Let $g(C)=  u_{i_1 } \cdots u_{i_k} $ for a cycle $C=( e_{i_1 }, \cdots, e_{i_k} )$. 
Then the {\em edge zeta function} $\zeta {}_G ({\bf u} )$ of $G$ is defined by 
\[
\zeta {}_G ( {\bf u} )= \prod_{[C]} (1-g(C) )^{-1} , 
\]
where $[C]$ runs over all equivalent classes of prime, reduced cycles of $G$.

\begin{thm}[Stark and Terras\cite{ST1996}]
Let $G$ be a connected graph with $m$ edges. Then  
\[
\zeta {}_G ( {\bf u} ) {}^{-1} = 
\det ( {\bf I}_{2m} - ( {\bf B} - {\bf J} {}_0) {\bf U} ) 
= \det ( {\bf I}_{2m} - {\bf U} ( {\bf B} - {\bf J} {}_0)) ,  
\]
where ${\bf U} ={\rm diag} (u_1 , \cdots , u_{2m} )$ is the diagonal matrix. 
\end{thm} 

In \cite{ST1996}, the multivariable analogue $N_k({\bf u})$ of $N_k$ was defined and 
used to obtain the determinant expression. 
Mizuno and Sato \cite{MS2004} defined the weighted zeta function of a graph $G$, and 
gave a determinant expression by the weighted matrix of $G$. 
Later, they called this zeta function the first weighted zeta function in order to 
distinguish it from another zeta function which Sato defined after. 
Assigning a complex number $w(e)$ to each arc $e{\;\in\;}D(G)$, 
we define 
${\bf W} =(w_{uv})_{u,v \in V(G)}$ to be the $n \times n$ matrix as follows: 
\[
w_{uv} =\left\{
\begin{array}{ll}
w(e) & \mbox{if $e=(u,v) \in D(G)$, } \\
0 & \mbox{otherwise, }
\end{array}
\right.
\]
${\bf W}$ is called the {\em weighted matrix} of $G$. 
For each path $P=(e_{i_1},e_{i_2},{\cdots},e_{i_d})$, 
the {\it norm} $w(P)$ of $P$ is defined by 
$w(P)=w(e_{i_1})w(e_{i_1}){\cdots}w(e_{i_d})$. 
The {\em {\rm (}first{\rm )} weighted zeta function} of $G$ is defined by 
\begin{equation}\label{DefFirstZeta}
{\bf Z} (G,w,t)= \prod_{[C]} (1-w(C) t^{ | C | } ) {}^{-1} , 
\end{equation}
where $[C]$ runs over all equivalent classes of prime, reduced cycles of $G$.

\begin{thm}[Mizuno and Sato\cite{MS2004}] 
Let $G$ be a connected graph with $n$ vertices.  
Suppose that $w(e^{-1} )= w (e )^{-1} $ for $e \in D(G)$. Then 
\[
{\bf Z} (G,w,t)^{-1}  
=(1- t^2 )^{m-n} \det ( {\bf I}_n -t {\bf W} + t^2 ( {\bf D} - {\bf I}_n )) . 
\]
${\bf Z} (G,w,t)$ has an expression as the exponential of a generating 
function as follows: 
\[
{\bf Z} (G,w,t)=\exp\Big{(}\sum_{[C]}\sum_{s{\geq}1}
\dfrac{1}{s}w(C)^st^{|C|s}\Big{)}, 
\] 
where $[C]$ runs over all equivalent classes of prime, reduced cycles of $G$. 
\end{thm}

Later, Watanabe and Fukumizu \cite{WF2009} gave a determinant expression for 
the edge zeta function of a graph with no additional condition $w(e^{-1} )= w (e )^{-1}$. 
Rearranging arcs so that $D(G)=\{e_1, e_1^{-1},\cdots,e_m, e_m^{-1}\}$, and putting 
$u_i=w(e_i)$, ${\bf U}$ turns out to be  
\begin{equation*}
{\bf U}=\begin{bmatrix}
w(e_1) & & & & 0 \\
 & w(e_{1}^{-1}) & & & \\
 & & \ddots & & \\
 & & & w(e_m) & \\
0& & & & w(e_{m}^{-1})
\end{bmatrix}.
\end{equation*} 
In \cite{WF2009}, two $n \times n$ matrices 
$\hat{\bf A}=(\hat{\bf A}_{uv})_{u,v{\in}V(G)}$ and 
$\hat{\bf D}=(\hat{\bf D}_{uv})_{u,v{\in}V(G)}$ were defined as follows: 
\begin{equation*}
\begin{split}
\hat{\bf A}_{uv} &= \begin{cases}
\dfrac{w(u,v)}{1-w(u,v)w(v,u)} & \mbox{if $(u,v){\;\in\;}D(G)$, } \\
0 & \mbox{otherwise,}
\end{cases}\\
\hat{\bf D}_{uv} &= \delta_{uv}
\sum_{\substack{e{\in}D(G)\\o(e)=u}}\dfrac{w(e)w(e^{-1})}{1-w(e)w(e^{-1})}.
\end{split}
\end{equation*}
where $\delta_{uv}$ denotes the Kronecker's delta. 
They derived their determinant expression with $\hat{\bf A}$ and $\hat{\bf D}$ 
instead of the weighted matrix and the degree matrix of $G$. 
\begin{thm}[Watanabe and Fukumizu\cite{WF2009}] 
\[
\zeta {}_G ( {\bf u} ) {}^{-1}
= \det ( {\bf I}_{2m} - {\bf U} ( {\bf B} - {\bf J} {}_0)) 
= \det({\bf I}_{n}+\hat{\bf D}-\hat{\bf A})\prod_{i=1}^m(1-w(e_i)w(e_i^{-1}))
\]
\end{thm}

\section{The quaternionic weighted zeta function of a graph} 
We continue with the notations in former sections. 
We shall define the quaternionic weighted zeta function of a graph. 
First we extend $w(e)$ to a quaternion. 
Let $w$ be a map from $V(G){\times}V(G)$ to $\HM$ such that 
$w(u,v)=0$ if $(u,v){\;\notin\;}D(G)$. We may write $w(e)$ instead of $w(u,v)$ 
if $e=(u,v){\;\in\;}D(G)$. We call $w(e)$ the {\it quaternionic weight} of $e{\;\in\;}D(G)$. 
For a path $P=( e_1 , \cdots , e_{\ell} )$ of $G$, the {\em norm} (or the {\em quaternionic 
weight}) $w(P)$ of $P$ is defined by $w(P)=w(e_1)w(e_2){\cdots}w(e_{\ell})$. 
We arrange the arcs $e_1,e_2,{\cdots},e_{2m}$ such that $e_{m+k}=e_{k}^{-1}$ for 
$k=1,{\cdots},m$. 
Now, we define the {\it quaternionic weighted zeta function} of $G$ to be the element 
of $\HM[[t]]$ as follows: 
\begin{equation}\label{DefQuatFirstZeta}
{\bf Z}_{\HM}(G,w,t)=\prod_{C}\Big{\{}(1-w(C)t^{|C|})(1-w(C)t^{|C|})^*\Big{\}}^{-1},
\end{equation}
where $C$ runs over all reduced cycles $C=e_{i_1}e_{i_2}{\cdots}e_{i_r}$ such that 
$i_1i_2{\cdots}i_r{\in}L_{[2m]}$. 
If $w(e)=1$ for any $e{\;\in\;}D(G)$, then the quaternionic weighted zeta function of 
$G$ is reduced to the square of the Ihara zeta function of $G$. 

\begin{rem}\label{RemQuatZeta}
\begin{enumerate}
\renewcommand{\labelenumi}{\rm (\roman{enumi})}
\item
The primeness is not required explicitly since 
$i_1i_2{\cdots}i_r{\in}L_{[2m]}$ implies the primeness of $e_{i_1}e_{i_2}{\cdots}e_{i_r}$. 
\item
{\rm (\ref{DefQuatFirstZeta})} does not depend on the order in which inverses of 
$(1-w(C)t^{|C|})(1-w(C)t^{|C|})^*$ are multiplied since 
$(1-w(C)t^{|C|})(1-w(C)t^{|C|})^*{\;\in\;}\CM[[t]]$ by Lemma \ref{ProdOfFAndConjF}. 
Indeed, the inverses are in $\RM[[t]]$ since 
$(1-w(C)t^{|C|})(1-w(C)t^{|C|})^*=1-2\{{\RP} w(C)\}t^{|C|}+|w(C)|^2t^{2|C|}$. 
\end{enumerate}
\end{rem}

We shall show that ${\bf Z}_{\HM}(G,w,t)$ can be expressed 
as the exponential of a generating function. 
We recall that Theorem \ref{ThmExpAndLog}, Proposition \ref{ProExp}, 
Proposition \ref{ProLog} and Corollary \ref{CorLog} hold even if we replace 
$\QM[X^*][[t]]$ with $\HM[[t]]$. By Remark \ref{RemQuatZeta}, we have: 
\begin{equation*}
{\bf Z}_{\HM}(G,w,t)=\prod_{C}\Big{\{}1-2\{{\RP} w(C)\}t^{|C|}+|w(C)|^2t^{2|C|}
\Big{\}}^{-1}.
\end{equation*}
Since $\alpha_C=2\{{\RP} w(C)\}t^{|C|}-|w(C)|^2t^{2|C|}$ is an element of 
$\RM[[t]]$ and has zero constant term, 
\begin{equation*}
\{1-2\{{\RP} w(C)\}t^{|C|}+|w(C)|^2t^{2|C|}\}^{-1}=\{1-\alpha_C\}^{-1}
=1+\sum_{k{\geq}1}\alpha_C^k
\end{equation*}
is also an element of $\RM[[t]]$ and $\beta_C=\sum_{k{\geq}1}\alpha_C^k$ has 
zero constant term. 
Since the constant term of ${\bf Z}_{\HM}(G,w,t)$ is equal to $1$, 
we can derive the following: 
\begin{equation*}
\begin{split}
\log{\bf Z}_{\HM}(G,w,t)&=\log\prod_{C}(1+\beta_C)=\sum_{C}\log(1+\beta_C)\\
&=\sum_{C}\log((1-\alpha_C)^{-1})=-\sum_{C}\log(1-\alpha_C)\\
&=-\sum_{C}\log\Big{\{}(1-w(C)t^{|C|})(1-w(C)^*t^{|C|})\Big{\}}\\
&=-\sum_{C}\Big{\{}\log(1-w(C)t^{|C|})+\log(1-w(C)^*t^{|C|})\Big{\}}\\
&=-\sum_{C}\sum_{n{\geq}1}\dfrac{-1}{n}
\big{\{}w(C)^n+(w(C)^*)^n\big{\}}t^{n|C|}\\
&=\sum_{C}\sum_{n{\geq}1}\dfrac{2{\RP}(w(C)^n)}{n}t^{n|C|}
=\sum_{n{\geq}1}\sum_{C}\dfrac{2{\RP}(w(C)^n)}{n}t^{n|C|}.
\end{split}
\end{equation*}
Therefore $\log{\bf Z}_{\HM}(G,w,t)$ has zero constant term, and it follows that: 
\begin{thm}\label{ThmExpGenFunction}
${\bf Z}_{\HM}(G,w,t)$ has the expression as the 
exponential of a generating function as follows: 
\begin{equation*}
{\bf Z}_{\HM}(G,w,t)=\exp\Big{(}\sum_{n{\geq}1}\sum_{C}
\dfrac{2{\RP}(w(C)^n)}{n}t^{n|C|}
\Big{)}.
\end{equation*}
where $C$ runs over all reduced cycles $C=e_{i_1}e_{i_2}{\cdots}e_{i_r}$ such that 
$i_1i_2{\cdots}i_r{\in}L_{[2m]}$. 
\end{thm}

We can also express ${\bf Z}_{\HM}(G,w,t)$ as the square of the 
exponential of a generating function as follows: 
\[
{\bf Z}_{\HM}(G,w,t)=\Big{\{}\exp\Big{(}\sum_{n{\geq}1}\sum_{C}
\dfrac{{\RP}(w(C)^n)}{n}t^{n|C|}\Big{)}\Big{\}}^2.
\]

\section{Study determinant expressions for ${\bf Z}_{\HM}(G,w,t)$}
Consider an $n \times n$ quaternionic matrix 
${\bf W} =({\bf W}_{uv})_{u,v{\in}V(G)}{\;\in\;}\Mat(n,\HM)$ with $(u,v)$-entry 
equals $0$ if $(u,v){\;\notin\;}D(G)$. 
We call ${\bf W}$ a {\em quaternionic weighted matrix} of $G$.
Furthermore, let $w(u,v)= {\bf W}_{uv}$ for $u,v \in V(G)$ and 
$w(e)= w(u,v)$ if $e=(u,v) \in D(G)$.

For a quaternionic weighted matrix ${\bf W}$ of $G$, 
we define two $2m \times 2m$ matrices 
${\bf B}_w=( {\bf B}^{(w)}_{ef} )_{e,f \in D(G)} $ and 
${\bf J}_w=( {\bf J}^{(w)}_{ef} )_{e,f \in D(G)} $ as follows: 
\[
{\bf B}^{(w)}_{ef} =\left\{
\begin{array}{ll}
w(e) & \mbox{if $t(e)=o(f)$, } \\
0 & \mbox{otherwise, } 
\end{array}
\right.
\  
{\bf J}^{(w)}_{ef} =\left\{
\begin{array}{ll}
w(e) & \mbox{if $f=e^{-1} $, } \\
0 & \mbox{otherwise. }
\end{array}
\right.
\]
We call ${\bf B}_w-{\bf J}_w$ the {\it quaternionic weighted edge matrix}. 
Then the Study determinant expression of ${\bf Z}_{\HM}(G,w,t)$ of Hashimoto type 
by use of the quaternionic weighted edge matrix is stated as follows: 

\begin{thm}\label{DetExpressionByEdgeMatrix}
\[
{\bf Z}_{\HM}(G,w,t)^{-1}= {\Sdet_t} ( {\bf I}_{2m} -( {\bf B}_w - {\bf J}_w )t) . 
\]
\end{thm}

Beforehand, we mention that the right hand side does not depend on the way of 
arranging arcs which determines row and column indexes. 
Indeed, if we rearrange arcs, then it turns into 
$\Sdet_t\big{(}{\bf P}( {\bf I}_{2m} -( {\bf B}_w - {\bf J}_w )t){\bf P}^{-1}\big{)}$
for some permutation matrix ${\bf P}$. However, from Proposition \ref{SdettProperties} 
(v), this is equal to $\Sdet_t( {\bf I}_{2m} -( {\bf B}_w - {\bf J}_w )t)$. 

\begin{proof} 
Set $A=\HM$ in Proposition \ref{AmitsurForMatrixThm}. 
Then we can take Study determinants of both sides in (\ref{EqnAmitsurIdentity}), so that  
\begin{equation}\label{EqnEulerProd}
\begin{split}
&\Sdet_t({\bf I}_{2m}-{\bf A}t)\\
&=\Sdet_t\Big{(}\prod_{\substack{(i_1,j_1){\cdots}(i_r,j_r){\in}L_{[2m]{\times}[2m]}\\
j_k=i_{k+1}\;(k=1,{\cdots},r-1)}}^{<}({\bf I}_n-a_{i_1i_2}a_{i_2i_3}{\cdots}
a_{i_{r-1}i_r}a_{i_rj_r}{\bf E}_{i_1j_r}t^{r})\Big{)}\\
&=\prod_{\substack{(i_1,j_1){\cdots}(i_r,j_r){\in}L_{[2m]{\times}[2m]}\\
j_k=i_{k+1}\;(k=1,{\cdots},r-1)}}\Sdet_t({\bf I}_n-a_{i_1i_2}a_{i_2i_3}{\cdots}
a_{i_{r-1}i_r}a_{i_rj_r}{\bf E}_{i_1j_r}t^{r}).
\end{split}
\end{equation}
We notice that the last formula does not depend on the order in which 
$\Sdet_t$ are multiplied since $\Sdet_t$ take values in $\RM[[t]]$. 
It follows from Proposition \ref{SdettProperties} (vii) that if $j_r=i_1$, then 
\begin{equation*}
\begin{split}
&\Sdet_t({\bf I}_n-a_{i_1i_2}a_{i_2i_3}{\cdots}a_{i_ri_1}
{\bf E}_{i_1i_1}t^{r})\\
&=(1-a_{i_1i_2}a_{i_2i_3}{\cdots}a_{i_ri_1}t^{r})
(1-a_{i_1i_2}a_{i_2i_3}{\cdots}a_{i_ri_1}t^{r})^*,
\end{split}
\end{equation*}
and otherwise,  
\[
\Sdet_t({\bf I}_n-a_{i_1i_2}a_{i_2i_3}{\cdots}a_{i_rj_r}
{\bf E}_{i_1j_r}t^{r})=1.
\]
Putting ${\bf A}={\bf B}_w-{\bf J}_w$, then 
\[
a_{ij}=a_{e_ie_j}=\begin{cases}
w(e_i) & \text{if $t(e_i)=o(e_j)$ and $e_j{\neq}e_i^{-1}$},\\
0 & \text{otherwise}.
\end{cases}
\]
Therefore, (\ref{EqnEulerProd}) yields: 
\begin{equation*}
\begin{split}
&\Sdet_t({\bf I}_{2m}-({\bf B}_w-{\bf J}_w)t)\\
&=\prod_{\substack{(i_1,i_2){\cdots}(i_r,i_1){\in}L_{[2m]{\times}[2m]}\\
e_{i_1}e_{i_2}{\cdots}e_{i_r} : \text{reduced cycle}}}
\Sdet_t({\bf I}_n-w(e_{i_1})w(e_{i_2}){\cdots}w(e_{i_r}){\bf E}_{i_1i_1}t^{r})\\
&=\prod_{\substack{(i_1,i_2){\cdots}(i_r,i_1){\in}L_{[2m]{\times}[2m]}\\
e_{i_1}e_{i_2}{\cdots}e_{i_r} : \text{reduced cycle}}}
(1-w(e_{i_1}){\cdots}w(e_{i_r})t^{r})
(1-w(e_{i_1}){\cdots}w(e_{i_r})t^{r})^*.
\end{split}
\end{equation*}
Each Lyndon word $(i_1,i_2){\cdots}(i_r,i_1)$ in $L_{[2m]{\times}[2m]}$ 
corresponds to a Lyndon word 
$i_1i_2{\cdots}i_r$ in $L_{[2m]}$ bijectively. Hence we have: 
\begin{equation*}
\begin{split}
&\Sdet_t({\bf I}_{2m}-({\bf B}_w-{\bf J}_w)t)\\
&=\prod_{\substack{i_1i_2{\cdots}i_r{\in}L_{[2m]}\\
e_{i_1}e_{i_2}{\cdots}e_{i_r} : \text{reduced cycle}}}
(1-w(e_{i_1}){\cdots}w(e_{i_r})t^{r})
(1-w(e_{i_1}){\cdots}w(e_{i_r})t^{r})^*.
\end{split}
\end{equation*}
Remark \ref{RemQuatZeta} (ii) implies 
the inverse of the right hand side is ${\bf Z}_{\HM}(G,w,t)$. 
Hence we obtain the conclusion as desired. 
\end{proof}

Finally, we shall show the Study determinant expression of Bass type. 
We define $n \times n$ matrices $\tilde{{\bf W}} =(\tilde{{\bf W}}_{uv} )_{u,v \in V(G)} $ and 
$ \tilde{\bf D} =(\tilde{\bf D}_{uv})_{u,v \in V(G)}$ to be as follows:
\begin{equation*}
\begin{split}
\tilde{\bf W}_{uv}&=\begin{cases}
(1-w(e)w(e^{-1})t^2)^{-1}w(e) & \text{if $e=(u,v){\in}D(G)$} \\
0 & \text{otherwise,}
\end{cases}\\
\tilde{\bf D}_{uv}&=\delta_{uv}\sum_{\substack{e{\in}D(G)\\o(e)=u}}
(1-w(e)w(e^{-1})t^2)^{-1}w(e)w(e^{-1}).
\end{split}
\end{equation*}
Then the Study determinant expression of Bass type is stated as follows: 

\begin{thm}\label{DetExpressionOfBassType}
\begin{equation*}
\begin{split}
&{\bf Z}_{\HM}(G,w,t)^{-1}\\
&=\Sdet_t({\bf I}_n-t\Tilde{\bf W}+t^2\Tilde{\bf D})
\prod_{i=1}^m(1-w(e_i)w(e_i^{-1})t^2)(1-w(e_i)w(e_i^{-1})t^2)^*. 
\end{split}
\end{equation*}
\end{thm}

\begin{proof}  We readily see that: 
\begin{equation*}
\begin{split}
\Sdet_t({\bf I}_{2m}-t({\bf B}_w-{\bf J}_w))
&=\Sdet_t({\bf I}_{2m}+t{\bf J}_w-t{\bf B}_w) \\
&=\Sdet_t({\bf I}_{2m}-t{\bf B}_w({\bf I}_{2m}+t{\bf J}_w)^{-1})\\
&{\quad}{\times}\Sdet_t({\bf I}_{2m}+t{\bf J}_w). 
\end{split}
\end{equation*}
Now, let ${\bf K} =( {\bf K}_{ev} )$ ${}_{e \in D(G); v \in V(G)}$ and 
${\bf L} =( {\bf L}_{ev} )_{e \in D(G); v \in V(G)} $ be 
two $2m \times n$ matrices defined as follows: 
\[
{\bf K}_{ev} :=\left\{
\begin{array}{ll}
w(e) & \mbox{if $t(e)=v$, } \\
0 & \mbox{otherwise, } 
\end{array}
\right.
\qquad 
{\bf L}_{ev} :=\left\{
\begin{array}{ll}
1 & \mbox{if $o(e)=v$, } \\
0 & \mbox{otherwise. } 
\end{array}
\right.
\] 
Then one can easily check that:  
\[
{\bf K} {}^T\! {\bf L} = {\bf B}_w. 
\]
Thus, from Proposition \ref{SdettProperties} (viii), we get the following: 
\begin{equation}\label{EqnSdettFactorization}
\begin{split}
\Sdet_t({\bf I}_{2m}-t({\bf B}_w-{\bf J}_w))
&=\Sdet_t({\bf I}_{2m}-t{\bf K}{}^T\!{\bf L}({\bf I}_{2m}+t {\bf J}_w)^{-1})\\
&{\quad}{\times}\Sdet_t({\bf I}_{2m}+t{\bf J}_w) \\ 
&=\Sdet_t({\bf I}_{n}-t{}^T\!{\bf L}({\bf I}_{2m}+t{\bf J}_w)^{-1}{\bf K})\\
&{\quad}{\times}\Sdet_t({\bf I}_{2m}+t{\bf J}_w). 
\end{split}
\end{equation} 
Arranging arcs so that $D(G)=\{e_1, e_1^{-1},\cdots,e_m, e_m^{-1}\}$, we have 
from Proposition \ref{SdettProperties} (iv), (vii) that:
\begin{equation}\label{SdettByProd}
\begin{split}
&\Sdet_t({\bf I}_{2m}+t{\bf J}_w)=
\Sdet_t \begin{bmatrix}
1 & w(e_1)t & 0 & 0 & \cdots & & \\
w(e_1^{-1})t & 1 & 0 & 0 & \cdots & & \\
0 & 0 & 1 & w(e_2)t & 0 & & \\
0 & 0 & w(e_2^{-1})t & 1 & 0 & & \\
\vdots & \vdots & \vdots & 0 & \ddots & & 
\end{bmatrix}\\
&=\Sdet_t \begin{bmatrix}
1-w(e_1)w(e_1^{-1})t^2 & w(e_1)t & 0 & 0 & \cdots & & \\
0 & 1 & 0 & 0 & \cdots & & \\
0 & 0 & 1-w(e_2)w(e_2^{-1})t^2 & w(e_2)t & 0 & & \\
0 & 0 & 0 & 1 & 0 & & \\
\vdots & \vdots & \vdots & 0 & \ddots & & 
\end{bmatrix}\\
&=\prod_{i=1}^m(1-w(e_i)w(e_i^{-1})t^2)(1-w(e_i)w(e_i^{-1})t^2)^*.
\end{split}
\end{equation}
Let ${\bf X}(e)$ be the $2{\times}2$ matrix defined as follows: 
\[
{\bf X}(e)=\begin{bmatrix}
1 & w(e)t \\
w(e^{-1})t & 1
\end{bmatrix}.
\]
Then, the inverse of ${\bf X}(e)$ is given by 
\[
{\bf X}(e)^{-1}=\begin{bmatrix}
(1-w(e)w(e^{-1})t^2)^{-1} & -(1-w(e)w(e^{-1})t^2)^{-1}w(e)t \\
-(1-w(e^{-1})w(e)t^2)^{-1}w(e^{-1})t & (1-w(e^{-1})w(e)t^2)^{-1}
\end{bmatrix}.
\]
Therefore, 
\[
({\bf I}_{2m}+t{\bf J}_w)^{-1}=
\begin{bmatrix}
{\bf X}(e_1)^{-1} & {\bf O} & \cdots & {\bf O} \\
{\bf O} & {\bf X}(e_2)^{-1} &  & \vdots \\
\vdots &  & \ddots & {\bf O} \\
{\bf O} & \cdots & {\bf O} & {\bf X}(e_m)^{-1} 
\end{bmatrix}.
\]
If $e=(u,v){\;\in\;}D(G)$, then we have by a calculation that: 
\begin{equation*}
\begin{split}
({}^T\!{\bf L}({\bf I}_{2m}+t{\bf J}_w)^{-1}{\bf K})_{uv}
&=\sum_{f,f'{\in}D(G)}({}^T\!{\bf L})_{uf}\big{(}({\bf I}_{2m}+t{\bf J}_w)^{-1}\big{)}_{ff'}
{\bf K}_{f'v}\\
&=({}^T\!{\bf L})_{ue}\big{(}({\bf I}_{2m}+t{\bf J}_w)^{-1}\big{)}_{ee}
{\bf K}_{ev}\\
&=(1-w(e)w(e^{-1})t^2)^{-1}w(e).
\end{split}
\end{equation*}
If $u=v$, then we have: 
\begin{equation*}
\begin{split}
({}^T\!{\bf L}({\bf I}_{2m}+t{\bf J}_w)^{-1}{\bf K})_{uu}
&=\sum_{f,f'{\in}D(G)}({}^T\!{\bf L})_{uf}\big{(}({\bf I}_{2m}+t{\bf J}_w)^{-1}\big{)}_{ff'}
{\bf K}_{f'u}\\
&=\sum_{\substack{e{\in}D(G)\\o(e)=u}}
({}^T\!{\bf L})_{ue}\big{(}({\bf I}_{2m}+t{\bf J}_w)^{-1}\big{)}_{ee^{-1}}
{\bf K}_{e^{-1}u}\\
&=-\sum_{\substack{e{\in}D(G)\\o(e)=u}}(1-w(e)w(e^{-1})t^2)^{-1}w(e)w(e^{-1})t.
\end{split}
\end{equation*}
Hence, it follows that: 
\begin{equation}\label{EqnByUseOfWAndD}
{}^T\!{\bf L}({\bf I}_{2m}+t{\bf J}_w)^{-1}{\bf K}=\Tilde{\bf W}-t\Tilde{\bf D}.
\end{equation}
Combining (\ref{EqnSdettFactorization}), (\ref{SdettByProd}) and 
(\ref{EqnByUseOfWAndD}), we obtain: 
\begin{equation*}
\begin{split}
&\Sdet_t({\bf I}_{2m}-t({\bf B}_w-{\bf J}_w))\\
&=\Sdet_t({\bf I}_n-t\Tilde{\bf W}+t^2\Tilde{\bf D})
\prod_{i=1}^m(1-w(e_i)w(e_i^{-1})t^2)(1-w(e_i)w(e_i^{-1})t^2)^*.
\end{split}
\end{equation*}
Now, the assertion follows from Theorem \ref{DetExpressionByEdgeMatrix}. 
\end{proof}



\end{document}